\documentclass{article}
	\usepackage{amsmath}
	\usepackage{amsfonts}
	\usepackage{amssymb}
	\newtheorem{thm}{Theorem}
	\newtheorem{prop}[thm]{Proposition}

	\newenvironment{proof}{\noindent\sl Proof. \rm}{$\,_\Box$\par\medskip}
	\newcommand{\G}[1]{\operatorname{G}\left(#1\right)}
	\newcommand{\Gs}[2]{\operatorname{G}_{#1}\left(#2\right)}
	
	\newcommand{\proj}[1]{\operatorname{Proj}\left(#1\right)}
\begin{document}
\title{Corrigendum to ``Reduced Tangent Cones\\ and Conductor at Multiplanar\\ Isolated Singularities''}
\author{Alessandro De Paris and Ferruccio Orecchia}
\date{}
\maketitle
\bigskip

The statement of \cite[Remark~2.2]{DO} is wrong (basically, a hypothesis is missed). That remark is used at the beginning of the proof of \cite[Theorem~3.2]{DO}:
\begin{equation}\label{P1}
\parbox{0.85\textwidth}{According to Proposition~2.1 and Remark~2.2, the hypotheses~1 and~3 imply that $\G{\overline{A}}$ is reduced.}
\end{equation}
It is also implicitly used in the middle of the proof:
\begin{equation}\label{P2}
\parbox{0.85\textwidth}{the hypothesis~3 in the statement implies that $\nu_n$ is an isomorphism for $n>>0$.}  
\end{equation}
In \eqref{P1}, the outcome $\proj{\G{A}}\cong\proj{\Gs{\overline{\mathfrak{m}}}{\overline{A}}}$ of the remark implies that $\G{\overline{A}}$ is reduced because of the hypothesis~1 in the statement and \cite[Proposition~2.1]{DO}. In \eqref{P2}, the fact that $\nu_n$ is an isomorphism for $n>>0$ is a direct consequence of $\overline{\mathfrak{m}}^n=\mathfrak{m}^n$ (and justifies the injectivity of $S\hookrightarrow\G{\overline{A}}$). Those outcomes hold true in the hypotheses of \cite[Theorem~3.2]{DO}, simply because by adding the hypothesis that $\proj{\G{A}}$ is reduced (hypothesis~1 in the theorem), \cite[Remark~2.2]{DO} becomes true. The explanation is given by Proposition~\ref{C1} below, which therefore can be used as a replacement for the wrong remark (provided that the hypothesis~1 is also mentioned in \eqref{P2}).

In the statement of Proposition~\ref{C1}, the main thesis is that $\nu_n$ is an isomorphism for $n>>0$. This also implies that $\mathfrak{m}^n=\overline{\mathfrak{m}}^n$, by the Nakayama's lemma, but this equality is not really needed (when we say `Then (2) is satisfied because $\mathfrak{m}^{n_0+i}=\mathfrak{J}^{n_0+i}$ for $i>>0$', (2) can easily be justified in another way).

\begin{prop}\label{C1}
Under the assumptions before \cite[Remark~2.2]{DO}, if $\proj{\G A}$ is reduced and $\sqrt{\mathfrak{b}}=\mathfrak{m}$, then $\nu_n$ is an isomorphism for $n>>0$, and in turn this implies that
\[
\proj{\G{A}}\cong\proj{\Gs{\overline{\mathfrak{m}}}{\overline{A}}}\;.
\]
\end{prop}
\begin{proof}
Since $A$ is noetherian and $\sqrt{\mathfrak{b}}=\mathfrak{m}$, we have $\mathfrak{m}^{n_0}\subseteq\mathfrak{b}$ for some $n_0$, hence $\overline{\mathfrak{m}}^{{n_0}}=\mathfrak{m}^{n_0}\overline{A}\subseteq\mathfrak{b}\overline{A}=\mathfrak{b}\subseteq\mathfrak{m}$.

If the class $x$ of a $f\in\mathfrak{m}^n\smallsetminus\mathfrak{m}^{n+1}$ in $\G A_n$ is in the kernel of $\nu_n$ that is, $f$ belongs to $\overline{\mathfrak{m}}^{n+1}$, we have
\[
f^{n_0}\in\overline{\mathfrak{m}}^{nn_0+n_0}=\mathfrak{m}^{nn_0+n_0}\overline{A}=\mathfrak{m}^{nn_0}\overline{\mathfrak{m}}^{n_0}\subseteq\mathfrak{m}^{nn_0+1}\;,
\]
hence $x^{n_0}=0\in\G A_{nn_0}$, that is, $x$ is nilpotent. Since $\proj{\G A}$ is reduced, there exists $n_1$ such that for all $n\ge {n_1}$ there is no nilpotent $x\in\G{A}_n\smallsetminus\{0\}$, and therefore $\nu_n$ is injective for all $n\ge n_1$.

Let us consider the powers $\mathfrak{m}^n$ and $\overline{\mathfrak{m}}^n$ as $A$-modules, and denote by $l(M)$ the length of an arbitrary $A$-module $M$. Note also that the graded components $\Gs{\overline{\mathfrak{m}}}{\overline{A}}_n=\overline{\mathfrak{m}}^n/\overline{\mathfrak{m}}^{n+1}$ are vector spaces over the residue field $k:=A/\mathfrak{m}$, because $\overline{\mathfrak{m}}^{n+1}=\mathfrak{m}\overline{\mathfrak{m}}^n$; the same is obviously true for the graded components of $\G A$. For every $n\ge n_0$ we have $\mathfrak{m}^n\subseteq\overline{\mathfrak{m}}^n\subseteq\overline{\mathfrak{m}}^{n_0}\subseteq\mathfrak{m}$, hence
\begin{equation}\label{sums}
\sum_{i=n_0}^{n-1}\dim_k\Gs{\overline{\mathfrak{m}}}{\overline{A}}_i=l\left(\frac{\overline{\mathfrak{m}}^{n_0}}{\overline{\mathfrak{m}}^n}\right)\le l\left(\frac{\mathfrak{m}}{\mathfrak{m}^n}\right)=\sum_{i=1}^{n-1}\dim_k{\G A}_i\;.
\end{equation}
Note that $\dim_k\Gs{\overline{\mathfrak{m}}}{\overline{A}}_i\ge\dim_k{\G A}_i$ for all $i\ge n_1$, because $\nu_i$ is injective for that values. Suppose now that it is not true that $\nu_n$ is an isomorphism for all $n>>0$. Then we can find as many values of $i$ as we want for which $\dim_k\Gs{\overline{\mathfrak{m}}}{\overline{A}}_i>\dim_k{\G A}_i$. This contradicts \eqref{sums} for a sufficiently large~$n$.

Finally, it is well known that if a graded ring homomorphism $S\to T$ preserving degrees induces isomorphisms on all components of sufficiently large degrees, then it induces an isomorphism $\proj T\overset\sim\to\proj S$.
\end{proof}

Next, \cite[Remark~2.2]{DO} is invoked in \cite[Section~5, p.~2977, lines~20--21]{DO}:
\begin{equation}\label{P3}
\parbox{0.85\textwidth}{Since $\sqrt{\mathfrak{b}}=\mathfrak{m}$ and $\mathfrak{J}=\mathfrak{m}\overline{A}$, we have $\proj{\G{A}}\cong\proj{\G{\overline{A}}}$ by Remark~2.2.}
\end{equation}
At that point, we do \emph{not} have that $\proj{\G{A}}$ is reduced, and this fact is part of n.~3 of Claim~5.1 (the only point of that Claim that was still to be proved). In what follows we assume the notation of \cite[Section~5]{DO}. To fix the mistake, the sentence \eqref{P3} above must be dismissed, but we can keep the fact that $\mathfrak{J}=\mathfrak{m}\overline{A}$ (which is true, because it had been proved earlier that $\mathfrak{J}=\mathfrak{m}B$ and $\overline{A}=B$).

Let us look at the subsequent discussion in \cite{DO}. First of all, the isomorphism $\G{\overline{A}}\cong\prod_{i=1}^e\Gs{\mathfrak{a}_i}{k[t,s]}$ is pointed out. Here, let us also denote by $\pi_i:\G{\overline{A}}\to\Gs{\mathfrak{a}_i}{k[t,s]}=k\left[\overline{t-a_{i1}},\overline{s}_i\right]$ the projection on the $i$th factor. Let us split as follows the natural homomorphism displayed at \cite[p.~2977, line~25]{DO}:
\begin{equation}\label{H}
k[X_0,\ldots ,X_r]\to\G A\overset{\nu}\longrightarrow\G{\overline{A}}\overset{\!\!\sim}\to\prod_{i=1}^e\Gs{\mathfrak{a}_i}{k[t,s]}\overset{\pi_i}\longrightarrow k\left[\overline{t-a_{i1}},\overline{s}_i\right]\;.
\end{equation}
(in the description $X_j\mapsto\overline{x_j}$ given in the paper, $\overline{x_j}$ denotes the class of $x_j\in R\subset k[t,s]$ in the degree one component of $\Gs{\mathfrak{a}_i}{k[t,s]}$). Assuming that $X_0,\ldots ,X_r$ act as the coordinate functions on $k^{r+1}$, and the linear forms in $k[X_0,\ldots ,X_r]$ act accordingly, for each choice of distinct $i_0,i_1\in\left\{1,\ldots ,e\right\}$, since $l_{i_0}$ and $l_{i_1}$ are skew lines we can fix linear forms $T_{i_0i_1}$,  $S_{i_0i_1}$ such that
\begin{itemize}
\item $T_{i_0i_1}$ vanishes on $\mathbf{a}_{i_0}$, $\mathbf{b}_{i_0}$, $\mathbf{b}_{i_1}$, and takes value $1/\rho_{i_1}$ on $\mathbf{a}_{i_1}$, with $\rho_{i_1}$ being as in \cite[Footnote~4]{DO};
\item $S_{i_0i_1}$ vanishes on $\mathbf{a}_{i_0}$, $\mathbf{b}_{i_0}$, $\mathbf{a}_{i_1}$, and takes value $1$ on $\mathbf{b}_{i_1}$.
\end{itemize}
Taking into account \cite[Footnote~4]{DO}, the images of $T_{i_0i_1}$ and $S_{i_0i_1}$ in $k\left[\overline{t-a_{i1}},\overline{s}_i\right]$ through \eqref{H} vanish for $i=i_0$, and equal $\overline{t-a_{i1}}$ and $\overline{s}_i$, respectively, for $i=i_1$. It follows that the image of $T_{1i}\cdots T_{(i-1)i}T_{(i+1)i}\cdots T_{ei}$ in $\prod_{i=1}^e\Gs{\mathfrak{a}_i}{k[t,s]}$ through \eqref{H} is $\left(0,\ldots,\overline{t-a_{i1}}^{e-1},0,\ldots,0\right)$, where the nonzero component occurs at the $i$-th place. In a similar way, every element of the form $\left(0,\ldots,\overline{t-a_{i1}}^a\overline{s}^b,0,\ldots,0\right)$, with $a+b\ge e-1$, is the image of a product of $a+b$ linear forms, suitably chosen among the $T_{i_0i_1}$s and the $S_{i_0i_1}$s. It follows that the homomorphism $k[X_0,\ldots ,X_r]\to\prod_{i=1}^e\Gs{\mathfrak{a}_i}{k[t,s]}$ is surjective in degrees $\ge e-1$. Hence $\nu_{d}$ is surjective for $d\ge e-1$, and is in fact an isomorphism by \cite[Lemma~3.1]{DO}.
Thus $\nu$ induces the required isomorphism $\proj{\G{A}}\cong\proj{\G{\overline{A}}}$.

Finally, let us take this occasion to also fix a few typos and mild issues.

\begin{enumerate}
\item
In the summarized description of the main result \cite[Theorem~3.2]{DO} given in \cite[Introduction]{DO} at the beginning of p.~2970, it is missed the hypothesis (duly reported in the statement of the theorem and in the abstract) that $\proj{\G{A}}$ must be reduced.

\item
Typo at [p.~2971, lines 12--14]: ``$\mathbb{P}^{r+1}:=\proj{k[X_1,\ldots ,X_k]}\ldots$ subscheme of $\mathbb{P}^{r+1}$'' should be replaced with ``$\mathbb{P}^r:=\proj{k[X_0,\ldots ,X_r]}\ldots$ subscheme of $\mathbb{P}^r$''.

\item
At the beginning of the proof of \cite[Theorem~4.1]{DO} one finds
\[
\parbox{0.85\textwidth}{The hypothesis~1 states, in particular, that $Y:=\proj{\G{A}}$ is reduced. Then, taking also into account the hypothesis~2, we immediately deduce from the results of Section~2 that $H\left(\overline{A},n\right)=(n+1)e$.}
\]
Since $\proj{\G{A}}$ is reduced, Proposition~\ref{C1} can serve as a replacement for \cite[Remark~2.2]{DO} in this situation, too.

\item
In \cite[Section~5]{DO} (main example), $e$ must be assumed $\ge 2$ and $n$ has to be replaced with $r$ in some lists such as $x_0,\ldots , x_n$,  $f_2,\ldots,f_n$, $h_2,\ldots, h_n$ and $X_0,\ldots ,X_n$, as well as in the exponent $n+1$ in \cite[Equation~(6)]{DO}.

In the sentence below Equation~(6), `such that $g(\overline{t})$ takes one of the above mentioned special values' should be completed with `or $g'(\overline{t})=0$'. Later, `vanishes in the ring $T:=B_f\otimes_{A_f}B_f$' should be `vanishes over the ring $T:=B_f\otimes_{A_f}B_f$' (that is, each component of the $(n+1)$-tuple under consideration vanishes in $T$).
\end{enumerate}

\end{document}